\documentclass[12pt]{amsart}
\usepackage{amssymb}
\usepackage{latexsym,ulem}
\usepackage[mathscr]{eucal}
\usepackage{color,colordvi,graphicx}
\usepackage{epsf}
\usepackage{epsfig}
\usepackage{color}
\DeclareGraphicsExtensions{.pstex,.eps,.epsi,.ps}
\newtheorem{theor}{~~~~Theorem}

\newtheorem{prop}{~~~~Proposition}
\newtheorem{cor}{~~~~Corollary}
\newtheorem{lemma}{~~~~Lemma}

\allowdisplaybreaks
\newtheorem{remark}{~~~~Remark}
\newtheorem{defin}{~~~~Definition}

 \newcommand{\Real}{\mathbb{R}}
 \newcommand{\Complex}{\mathbb{C}}

 \newcommand{\hor}{\mathbf{hor}}
 \newcommand{\ver}{\mathbf{ver}}

  \newcommand{\di}{\textbf{div}}

 \newcommand{\tr}{\textbf{tr}}
 \newcommand{\D}{\mathfrak D}
 \newcommand{\Di}{\mathcal D}
 \newcommand{\ham}{\mathbf H}
 \newcommand{\x}{\mathbf x}
 \newcommand{\p}{\mathbf p}
  
 \newcommand{\tv}{\mathbf v}
 
 \newcommand{\MCP}[2]{\mathcal{MCP}(#1,#2)}
 \newcommand{\gMCP}[3]{\mathcal{MCP}(#1; #2, #3)}
 \newcommand{\hMCP}[4]{\mathcal{MCP}(#1, #2; #3, #4)}
 \newcommand{\Rm}{\mathbf{Rm}}

\newcommand{\vol}{\text{vol}}

\begin{document}
\bibliographystyle{plain}
\title[
Ricci curvature
type lower bounds]{Ricci curvature
type lower bounds for sub-Riemannian structures on Sasakian manifolds
}
\author{Paul W. Y. Lee}
\address{Room 216, Lady Shaw Building, The Chinese University of Hong
Kong, Shatin, Hong Kong}
\email{wylee@math.cuhk.edu.hk}
\author{Chengbo Li}
\address{Department of
Mathematics, Tianjin University, Tianjin, 300072, P.R.China}
\email{chengboli@tju.edu.cn}
\author{Igor Zelenko}
\address{Department of Mathematics, Texas A$\&$M University,
   College Station, TX 77843-3368, USA}
   \email{zelenko@math.tamu.edu}
\thanks{The authors would like to thank the referees for their constructive comments. The first author's research was supported by the Research Grant Council of Hong Kong (RGC Ref. No. CUHK404512). The second author was supported in part by the National Natural Science Foundation of China (Grant No. 11201330)}.

\newcounter{pl}
\newcommand{\pl}[1]
{\stepcounter{pl}$^{\bf PL\thepl}$%
\footnotetext{\hspace{-3.7mm}$^{\blacksquare\!\blacksquare}$
{\bf PL\thepl:~}#1}}

\begin{abstract}
Measure contraction properties are generalizations of the notion of Ricci curvature lower bounds in Riemannian geometry to more general metric measure spaces. In this paper, we give sufficient conditions for a Sasakian manifold equipped with a natural sub-Riemannian distance to satisfy these properties. Moreover, the sufficient conditions are defined by the Tanaka-Webster curvature. This generalizes the earlier work in \cite{AgLe1} for the three dimensional case and in \cite{Ju} for the Heisenberg group. To obtain our results we use the intrinsic Jacobi equations along sub-Riemannian extremals, coming from the theory of canonical moving frames for curves in Lagrangian Grassmannians \cite{LiZe1, LiZe2}. The crucial new tool here is a certain decoupling of the corresponding matrix Riccati equation. It is also worth pointing out that our
method leads to exact formulas for the measure contraction in the case of the corresponding homogeneous models in the considered class of sub-Riemannian  structures.
\end{abstract}
\date{\today}

\maketitle

\section{Introduction}

In recent years, there are lots of efforts in generalizing the notion of Ricci curvature lower bounds in Riemannian geometry and its consequences to more general metric measure spaces. One of them is the work of \cite{LoVi1, LoVi2, St1, St2} where a notion of synthetic Ricci curvature lower bound
was introduced using the theory of optimal transportation. These conditions are generalizations of Ricci curvature lower bounds (and curvature-dimension conditions introduced in \cite{BaEm}) to length spaces equipped with a measure (length spaces are metric spaces on which the notion of geodesics is defined). In \cite{Oh2}, it was shown that this synthetic Ricci curvature lower bound coincides with the pre-existing notion of Ricci curvature lower bounds in the case of Finsler manifolds.

On the contrary, it was shown in \cite{Ju} that this synthetic Ricci curvature lower bound defined using the theory of optimal transportation are not satisfied on the Heisenberg group, the simplest sub-Riemannian manifold (Note however that a type of curvature-dimension conditions were defined in \cite{BaGa1, BaGa2} using a  sub-Riemannian version of the Bochner formula). It was also shown in \cite{Ju} that the Heisenberg group satisfies another generalization of Ricci curvature lower bounds to length spaces called measure contraction properties \cite{St1, St2, Oh1}.

Measure contraction properties are essentially defined by the rate of contraction of volume along geodesics inspired by the classical Bishop volume comparison theorem. For the Riemannian manifold of dimension $n$, the measure contraction property $\MCP k {n}$ is equivalent to the conditions that the Ricci curvature is bounded below by $k$.
 In \cite{Ju}, it was shown that  the left-invariant sub-Riemannian metric on the Heisenberg group of dimension $2n+1$ satisfies the condition $\MCP 0 {2n+3}$. Such sub-Riemannian metrics can be regarded as the flat one among all sub-Riemannian metrics on Sasakian manifolds of the same dimension.

The next natural task is to study the measure contraction property for general (curved) sub-Riemannian metrics on Sasakian manifolds and, in particular, to understand what differential invariants of such metrics are important for their measure contraction property. A natural way of doing this is to analyze  the Jacobi equation along a sub-Riemannian extremal. In order to write the Jacobi equation intrinsically  one needs first to construct a connection canonically associated with a geometric structure. The construction of such connection is known in several classical cases such as the Levi-Civita connection for Riemannian metrics, the Tanaka-Webster connection for a special class of sub-Riemannian contact metrics associated with CR structures (\cite{tan, web})  or its generalization, the Tanno connection, to more general class of sub-Riemannian contact metrics associated with (non-integrable) almost CR structures (\cite{tanno}). However, these constructions use specific properties of the geometric structures under consideration  and it is not clear how to generalize them to more general sub-Riemannian structures.

An alternative  approach for obtaining intrinsic Jacobi equation {\it without preliminary construction of a canonical connection} was initiated in \cite{agrgam1} and further developed in \cite{jac1, LiZe1, LiZe2}. In this approach one replaces the Jacobi equation along an extremal by a special and a priori intrinsic curve of Lagrangian subspaces in a linear symplectic space, i.e. a curve in a Lagrangian Grassmannian. This curve is defined up to the natural action of the linear symplectic group. It contains all information about the space of Jacobi fields along the extremal and therefore it is called the {\it Jacobi curve of the extremal}.

By analogy with the classical Frenet-Serret frame for a curve in an Euclidean space, one can construct a bundle of canonical moving symplectic frames for a curve in a Lagrangian Grassmannian satisfying very general assumptions \cite{LiZe1, LiZe2}. The structure equations for these moving frames can be considered as the intrinsic Jacobi equations, while the nontrivial entries in the matrices of these structure equations give the invariants of the original geometric structures. These invariants can be used in principal for obtaining various comparison type results including the measure contraction properties. Although the construction of these invariants is algorithmic, to express them explicitly in terms of the original geometric structure is not an easy task already for sub-Riemannian contact case (\cite{AgLe1, LiZe3, LeLi}). Besides, in contrast to the Riemannian case,  the level sets of a sub-Riemannian Hamiltonian are not compact. Therefore, to control the bounds for the symplectic invariants of the Jacobi curves along extremals, additional assumptions have to be imposed.

For the first time the scheme based on the geometry of Jacobi curves was used in the study of the measure contraction properties  in \cite{AgLe1} (see also \cite{Hu,ChYa,AgLe2} for closely related results), where  general three dimensional contact sub-Riemannian manifolds were treated. In particular,  it was shown there  that the measure contraction properties $\MCP 0 5$ are characterized by the Tanaka-Webster sectional curvature of the planes which are the fibers of the contact distribution. More precisely, the sub-Riemannian metric naturally associated with a three-dimensional Sasakian manifold satisfies $\MCP 0 5$ if and only if the corresponding Tanaka-Webster sectional curvature is bounded below by 0.

Moreover, the generalized measure contraction properties $\gMCP{k}{2}{3}$ were defined there. It was shown that the sub-Riemannian metric naturally associated with a three-dimensional Sasakian manifold satisfies this condition if and only if the aforementioned Tanaka-Webster sectional curvature  is bounded below by $k$. The reason for appearance of $(2,3)$ in the notation $\gMCP{k}{2}{3}$ is to emphasize that this property holds for a class of sub-Riemannian metrics on rank $2$ distributions in $3$-dimensional manifolds.

In this paper, we generalize the results in \cite{Ju, AgLe1} on sub-Riemannian metrics naturally associated with a three-dimensional Sasakian manifold  to sub-Riemannian metrics associated with Sasakian manifolds of arbitrary (odd) dimension. We introduce new generalized measure contraction properties  $\hMCP{k_1}{k_2}{N-1}{N}$ (see Definition \ref{gMC} below) which are motivated by the measure contraction of the sub-Riemannian space forms (see Propositions \ref{maincor2}, \ref{maincor3}, and \ref{maincor4}) and discuss when a Sasakian manifold equipped with the sub-Riemannian structure mentioned above satisfies them.
The reason for appearance of $(N-1,N)$ in the notation $\hMCP{k_1}{k_2}{N-1}{N}$ is to emphasize that this property holds for a class of sub-Riemannian metrics on distributions of rank $N-1$ in $N$-dimensional manifolds (with odd $N\geq 3$). Note that for $N=3$ the property $\hMCP{k_1}{k_2}{N-1}{N}$ does not depend on $k_2$ and coincides with the property  $\gMCP{k_1}{2}{3}$ of \cite{AgLe1}.

To obtain these results we analyze the matrix Riccati equation associated with the structure equation for the canonical moving frame of the Jacobi curves (Lemma \ref{keylem} and section \ref{Ricsec} below) on the basis of  comparison theorems for matrix Riccati equations (\cite {Ro, Le}) and use the expressions for the symplectic invariants of the Jacobi curves (see section 5 below) based on the calculation in \cite{LeLi}.

The key new observation, which was first appeared in an earlier version of this paper, is the decoupling of these equations after taking the traces of appropriate blocks (as in the proof of Lemma \ref{Riccati2} below): the coupled equations after taking the traces yields to decoupled inequalities, which leads to the desired estimates, involving the Tanaka-Webster curvature of the manifold. Surprisingly, the passage from equations to inequalities does not affect the fact that we get exact formulas for the measure contraction in the case of the corresponding homogeneous models in the considered class of sub-Riemannian  structures. (see Propositions \ref{maincor2}, \ref{maincor3}, and \ref{maincor4} below).

Next, we state the condition $\hMCP{k_1}{k_2}{N-1}{N}$ and the main results in more detail. Let $M$ be a sub-Riemannian manifold (see Section \ref{subIntro} for a discussion on some basic notions in sub-Riemannian geometry). For simplicity, we assume that $M$ satisfies the following property: given any point $x_0$ in $M$, there is a subset of full measure on $M$  (i.e  a complement to a set of Lebesgue measure zero) such that any point of this subset is connected to $x_0$ by a unique length minimizing sub-Riemannian geodesic. By  \cite{CaRi}, this property holds  for all contact sub-Riemannian metrics due to the lack of  abnormal extremals (see Lemma \ref{clar1.3} of section \eqref{Ricattisec}).

For any point $x$ in this subset of full measure let $t\mapsto \varphi_t(x)$ be the unique geodesic starting from $x$ and ending at $x_0$. This defines a 1-parameter family of Borel maps $\varphi_t$. Let $d$ be the sub-Riemannian distance and let $\mu$ be a Borel measure. The following is the original measure contraction property studied in \cite{St1, St2, Oh1}:

A metric measure space $(M, d, \mu)$ satisfies $\MCP 0 N$ if
\[
\mu(\varphi_t(U))\geq (1-t)^N\mu(U)
\]
for each point $x_0$, each Borel set $U$, and for all $t$ in the interval $[0,1]$.

\begin{defin}
\label{gMC}
A metric measure space $(M, d, \mu)$ satisfies $\hMCP{k_1}{k_2}{N-1}{N}$ if, for each point $x_0$, each Borel set $U$, and each $t$ in the interval $[0,1]$,
\[
\mu(\varphi_t(U))\geq \int_{U}\frac{(1-t)^{N+2}\mathcal M_1(k_1d^2(x,x_0),t)\mathcal M_2^{N-3}(k_2d^2(x,x_0),t)}{\mathcal M_1(k_1d^2(x,x_0),0)\mathcal M_2^{N-3}(k_2d^2(x,x_0),0)}d\mu(x),
\]
where $\mathfrak D(k,t)=\sqrt{|k|}(1-t)$,
\[
\mathcal M_1(k,t)=\begin{cases}
\frac{2-2\cos(\mathfrak D(k,t))-\mathfrak D(k,t)\sin(\mathfrak D(k,t))}{\mathfrak D(k,t)^4} & \text{if } k>0\\
\frac{1}{12} & \text{if } k=0\\
\frac{2-2\cosh(\mathfrak D(k,t))+\mathfrak D(k,t)\sinh(\mathfrak D(k,t))}{\mathfrak D(k,t)^4} & \text{if } k<0,\\
\end{cases}
\]
and
\[
\mathcal M_2(k,t)=\begin{cases}
\frac{\sin(\mathfrak D(k,t))}{\mathfrak D(k,t)} & \text{if } k>0\\
1 & \text{if } k=0\\
\frac{\sinh(\mathfrak D(k,t))}{\mathfrak D(k,t)} & \text{if } k<0.
\end{cases}
\]
\end{defin}

Note, in particular, that $\hMCP 0 0 {N-1} N$ is the same as $\MCP 0 {N+2}$. If $k_1\geq 0$ and $k_2\geq 0$, then $\hMCP{k_1}{k_2}{N-1} N$ implies $\MCP 0 {N+2}$. The reason for the notations in the conditions $\mathcal{MCP}(k_1,k_2;N-1, N)$ is clarified by Theorem 1.1 below.

Next, we state a simple consequence of the main results. For this, let $M$ be a Sasakian manifold of dimension $2n+1$ equipped with a contact form $\alpha_0$ and a Riemannian metric $\left<\cdot,\cdot\right>$ (see the next section for a detail discussion of Sasakian manifold and the corresponding sub-Riemannian structure). Let $\mu$ be the corresponding Riemannian volume form and let $d_{CC}$ be the sub-Riemannian distance. Here the distribution is defined by $\Di=\ker\alpha_0$ and the sub-Riemannian metric is given by the restriction of the Riemannian one on $\Di$.

\begin{theor}\label{main}
Assume that the Tanaka-Webster curvature $\Rm^*$ of the Sasakian manifold $M$ of dimension $2n+1$ satisfies
\begin{equation}
\label{holomsect}
\left<\Rm^*(Jv,v)v,Jv\right>\geq k_1
\end{equation}
and
\begin{equation}
\label{SasRicci}
\sum_{i=1}^{2n-2}\left<\Rm^*(w_i,v)v,w_i\right>\geq (2n-2)k_2.
\end{equation}
for all unit tangent vectors $v$ in $\mathcal D$ and an orthonormal frame $\{w_i\}_{i=1}^{2n-2}$ of the space  $\mathcal D\cap\{v,Jv\}^\perp$.
Then the metric measure space $(M, d_{CC}, \mu)$ satisfies $\hMCP{k_1}{k_2}{2n}{2n+1}$, where $d_{CC}$ is the sub-Riemannian distance of $M$ and $\mu$ is the volume form with respect  to the Riemannian metric of the Sasakian manifold $M$.
\end{theor}

In particular, we recover the following result in \cite{Ju}.

\begin{theor}\cite{Ju}\label{cor1}
The Heisenberg group of dimension $2n+1$ equipped with the standard sub-Riemannian distance $d_{CC}$ and the Lebesgue measure $\mu$ satisfies $\hMCP{0}{0}{2n}{2n+1}=\MCP{0}{2n+3}$.
\end{theor}

We also have
\begin{theor}\label{cor2}
The complex Hopf fibration equipped with a natural sub-Riemannian distance $d_{CC}$ and measure $\mu$ (defined in Section \ref{subIntro}) satisfies the condition $\hMCP{4}{1}{2n}{2n+1}$. In particular, it satisfies $\MCP{0}{2n+3}$.
\end{theor}
and
\begin{theor}\label{cor3}
The anti de-Sitter space equipped with a natural sub-Riemannian distance $d_{CC}$ and measure $\mu$ (defined in Section \ref{subIntro}) satisfies the condition $\hMCP{-4}{-1}{2n}{2n+1}$.
\end{theor}

\begin{remark}\label{holomrem}
Note that the left hand side of \eqref{holomsect} is in fact the holomorphic sectional curvature of the K\"{a}hler structure obtained locally from the original Sasakian structure via the quotient by the transversal symmetry (for this point of view see our earlier preprint \cite{LLZ})  so that condition \eqref{holomsect} is in fact  a lower bound for the holomorphic sectional curvature of this K\"{a}hler structure. The inequality \eqref{SasRicci} plays a role of the  Sasakian Ricci curvature bound.
\end{remark}

\begin{remark}\label{k1k2}
Assume that the sectional curvature of the $2$-planes belonging to $\mathcal D$ of the Tanaka-Webster connection of the Sasakian manifold $M$  is bounded below by $k$. Then Theorem \ref{main} holds for $k_1=k_2=k$, i.e. the metric measure space $(M, d_{CC}, \mu)$ satisfies $\hMCP{k}{k}{2n}{2n+1}$.
\end{remark}

We also remark that the estimates for the proof of Theorem \ref{cor1}, \ref{cor2}, and \ref{cor3} are sharp (see Propositions \ref{maincor2}, \ref{maincor3}, and \ref{maincor4} for more detail).

Note that the condition $\MCP 0 N$ implies the volume doubling property of $\mu$ and a local Poincar\'e inequality  (see \cite{St1, St2, Oh1, AgLe1}, see also \cite{BaGa1,BaGa2, BaBoGa} for an alternative approach to the following results in the case $p=2$ with weaker assumptions).

\begin{cor}(Doubling)\label{double}
Assume that the assumption of Theorem \ref{main} holds with $k_1=k_2=0$. Then there is a constant $C>0$ such that
\[
 \mu(B_{x}(2R))\leq C\mu(B_{x}(R))
\]
for all $x$ in $M$ and all $R>0$, where $B_x(R)$ is the sub-Riemannian ball of radius $R$ centered at $x$.
\end{cor}

\begin{cor}(Poincar\'e inequality)\label{Poincare}
Assume that the assumption of Theorem \ref{main} holds with $k_1=k_2=0$. Then, for each $p>1$, there is a constant $C>0$ such that
\[
\begin{split}
&\int_{B_{x}(R)}|f(x)-\left<f\right>_{B_{x}(R)}|^pd\vol(x) \\
&\leq CR^p\int_{B_{x}(R)}|\nabla_{\hor}f|^pd\vol(x),
\end{split}
\]
where $\left<f\right>_{B_{x}(R)}=\frac{1}{\vol({B_{x}(R)})}\int_{B_{x}(R)}f(x)d\vol(x)$ and $\nabla_{\hor} f$ is the horizontal gradient of $f$ which is the projection of the Riemannian gradient onto the distribution $\mathcal D$.
\end{cor}

Let $\di_\eta$ be the divergence with respect to the volume form $\eta$. By combining Corollary \ref{double} and \ref{Poincare} with the results in \cite{CoHoSa}, we also obtain

\begin{cor}(Harnack inequality)
Assume that the assumption of Theorem \ref{main} holds with $k_1=k_2=0$. Then, for each $p>1$, there is a constant $C>0$ such that any positive solution to the equation $\di_\eta(|\nabla_{\hor}f|^{p-2}\nabla_{\hor}f)=0$ on $B_x(R)$ satisfies
\[
\sup_{B_{x}(R/2)}f\leq C\inf_{B_{x}(R/2)}f.
\]
\end{cor}

\begin{cor}(Liouville theorem)
Assume that the assumption of Theorem \ref{main} holds with $k_1=k_2=0$. Then any non-negative solution to the equation $\di_\eta(|\nabla_{\hor}f|^{p-2}\nabla_{\hor}f)=0$ on $M$ is a constant.
\end{cor}

For other consequences of Corollary \ref{double} and \ref{Poincare}, see \cite{CoHoSa}.

In the earlier longer versions of this paper (\cite{LLZ}), we prove the same results
for contact sub-Riemannain manifolds with a transversal symmetry, which are more general than the Sasakian one.
However, the formulation of the results in that more general cases are more cumbersome, while the proofs are practically the same. Therefore we decided to
omit this more general cases for the sake of more short and concise presentation. Besides, in the earlier version, even in the case of Sasakian structures,
we formulated the results in terms of the invariants of the K\"{a}hler structure obtained from the original Sasakian structure via the quotient by the
transversal symmetries, assuming that such quotient is globally defined. Here we use the calculations in \cite{LeLi} appeared later than the earlier version \cite{LLZ} to
rewrite the results in terms of the Tanaka-Webster curvature tensor of the Sasakian manifold itself. First, in this way the formulation becomes more compact
and natural. Second, the previous assumption on the global existence of the quotient manifold can be removed.


We also remark that, unlike the Riemannian case, Bishop volume comparison theorem and measure contraction properties are very different in the sub-Riemannian case. This is because any sub-Riemannian ball contains cut points.

In the next section, some basic notions in sub-Riemannian geometry will be recalled and the main results of the paper will be stated. The rest of the sections will be devoted to the proof of the main results.

\smallskip

\section{Sub-Riemannian Structures on Sasakian Manifolds}\label{subIntro}
\setcounter{equation}{0} \setcounter{theor}{0}
\setcounter{lemma}{0} \setcounter{prop}{0}

In this section, we recall various notions on sub-Riemannian structures of Sasakian manifolds which are needed. A sub-Riemannian manifold is a triple $(M,\Di,\left<\cdot,\cdot\right>)$, where $M$ is a manifold of dimension $N$, $\Di$ is a sub-bundle of the tangent bundle $TM$, and $\left<\cdot,\cdot\right>$ is a smoothly varying inner product defined on $\mathcal D$. The sub-bundle $\Di$ and the inner product $\left<\cdot,\cdot\right>$ are commonly known as a distribution and a sub-Riemannian metric, respectively. A Lipshitzian curve $\gamma(\cdot)$ is horizontal if $\dot\gamma(t)$ is contained in $\Di$ for almost every $t$. The length $l(\gamma)$ of a horizontal curve $\gamma$ can be defined as in the Riemannian case:
\[
l(\gamma)=\int_0^1|\dot\gamma(t)|dt.
\]

Assume that the distribution $\Di$ satisfies the following bracket generating or H\"ormander condition: the sections of $\Di$ and their iterated Lie brackets span each tangent space. Under this assumption and that the manifold $M$ is connected, the Chow-Rashevskii Theorem (see \cite{Mo}) guarantees that any two given points on the manifold $M$ can be connected by a horizontal curve. Therefore, we can define the sub-Riemannian (or Carnot-Carath\'eordory) distance $d_{CC}$ as
\begin{equation}\label{SRd}
d_{CC}(x_0,x_1)=\inf_{\gamma\in\Gamma}l(\gamma),
\end{equation}
where the infimum is taken over the set $\Gamma$ of all horizontal paths $\gamma:[0,1]\to M$ which connect $x_0$ with $x_1$: $\gamma(0)=x_0$ and $\gamma(1)=x_1$. The minimizers of (\ref{SRd}) are called length minimizing geodesics (or simply geodesics). As in the Riemannian case, reparametrizations of a geodesic is also a geodesic. Therefore, we assume that all geodesics have constant speed.

In this paper, we will focus on Sasakian manifolds equipped with a natural sub-Riemannian structure. First, we recall that a distribution $\Di$ is contact if it is given by the kernel of a 1-form $\alpha_0$, called a contact form, defined by the condition that the restriction of $d\alpha_0$ to $\Di$ is non-degenerate. A Sasakian manifold $M$ is a contact manifold equipped with a contact form $\alpha_0$, a vector field $v_0$, a $(1,1)$-tensor $J$, and a Riemannian metric $\left<\cdot,\cdot\right>$ such that
\begin{eqnarray}
&~&  \alpha_0(v_0)=1,\label{Sas1}\\
&~& J^2v=-v+\alpha_0(v)v_0,\label{Sas2}\\
&~& d\alpha_0(v,w)=\left<v,Jw\right>,\label{Sas3}
\end{eqnarray}
and the Nijenhuis tensor
\[
[J,J](v,w):=J^2[v,w]+[Jv,Jw]-J[Jv,w]-J[v,Jw]
\]
satisfies
\[
[J,J]=-d\alpha_0(v,w)v_0.
\]
The sub-Riemannian metric is defined simply by restricting the Riemannian one on the distribution $\mathcal D:=\ker\alpha_0$.

It follows from \cite[Proposition on p.20]{Blair} that conditions \eqref{Sas1} and \eqref{Sas2} imply
\begin{equation}
\label{Jv0}
Jv_0=0,
\end{equation}
each fiber of $\mathcal D$ is invariant with respect to $J$, and $J$ defines a complex structure on these fibers. Moreover, substituting  $w=v_0$ into \eqref{Sas3} and using \eqref{Jv0}, one gets that $v_0$ is contained in the kernel of $d\alpha_0$. This together with \eqref{Sas1} means that $v_0$ is the Reeb field of the contact form $\alpha_0$. Besides, putting $v=v_0$, using the fact that $v_0$ is contained in the kernel of $d\alpha_0$,
and that $J|_\Di$ is an isomorphism of $\mathcal D$, we get that $v_0$ is orthogonal to $\mathcal D$ with respect to the Riemannian metric $\left<\cdot,\cdot\right>$.

The Tanaka-Webster connection $\nabla^*$ is defined by
\[
\nabla_X^*Y=\nabla_XY+\frac{1}{2}\alpha_0(X)JY-\alpha_0(Y)\nabla_Xv_0+\nabla_X\alpha_0(Y)v_0,
\]
where $\nabla$ is the Levi-Civita connection of the Riemannian metric $\left<\cdot,\cdot\right>$. The corresponding curvature $\Rm^*$ is called Tanaka-Webster curvature
\[
\Rm^*(X,Y)Z=\nabla_X^*\nabla_Y^*Z-\nabla_Y^*\nabla_X^*Z-\nabla_{[X,Y]}^*Z.
\]

The simplest example of Sasakian manifold is the Heisenberg group. In this case, the manifold $M$ is the $(2n+1)$-dimensional Euclidean space $\Real^{2n+1}$ with coordinates $x_1,...,x_n,y_1,...,y_n,z$. The vector field $v_0$ is $v_0=\partial_z$, and the contact form $\alpha_0$ is $\alpha_0=dz-\frac{1}{2}\sum_{i=1}^nx_idy_i+\frac{1}{2}\sum_{i=1}^ny_idx_i$. Let
\[
X_i=\partial_{x_i}-\frac{1}{2}y_i\partial_z,\quad Y_i=\partial_{y_i}+\frac{1}{2}x_i\partial_z,
\]
where $i=1,...,n$.
The Riemannian metric is defined in such a way that $\{X_1,...,X_n,Y_1,...,Y_n,\partial_z\}$ is orthonormal and the tensor $J$ satisfies $J(X_i)=Y_i$, $J(Y_i)=-X_i$, $J(\partial_z)=0$. Finally, the Tanaka-Webster curvature $\Rm^*=0$ on $\mathcal D$ in this case. Remark that the Riemannian volume in this case coincides with the $(2n+1)$-dimensional Lebesgue measure $\mathcal L^{2n+1}$. The following three propositions are consequences of the \emph{proof} of Theorem \ref{main}.

\begin{prop}\label{maincor2}
The Heisenberg group satisfies
\[
\mathcal L^{2n+1}(\varphi_t(U))= \int_{U}\frac{(1-t)\mathcal M_1(\mathbf{k_1}(x),t)\mathcal M_2^{2n-2}(\mathbf{k_2}(x),t)}{\mathcal M_1(\mathbf{k_1}(x),0)\mathcal M_2^{2n-2}(\mathbf{k_2}(x),0)}d\mathcal L^{2n+1}(x)
\]
for any Borel set $U$, where
\[
\begin{split}
&\mathbf{k_1}(x)=(v_0\mathfrak f(x))^2, \quad \mathbf{k_2}(x)=\frac{(v_0\mathfrak f(x))^2}{4},\quad  \mathfrak f(\x)=-\frac{1}{2}d^2_{CC}(\x,\x_0).
\end{split}
\]
In particular, the metric measure space $(\mathbb{H}, d_{CC}, \mathcal L^{2n+1})$ satisfies the condition $\hMCP{0}{0}{2n}{2n+1}=\MCP{0}{2n+3}$.
\end{prop}

Another example which is relevant for us is the complex Hopf fibration. In this case, the manifold $M$ is the sphere $S^{2n+1}=\{z\in\Complex^{n+1}||z|=1\}$. The vector field $v_0$ is given by $v_0=2\sum_{i=1}^n(-y_i\partial_{x_i}+x_i\partial_{y_i})$ and the 1-form is $\alpha_0=\frac{1}{2}\sum_{i=1}^n(x_idy_i-y_idx_i)$. The Riemannian metric coincides with the Euclidean one on $\mathcal D$ and $J$ coincides with standard complex structure on $\mathcal D$. The Tanaka-Webster curvature $\Rm^*$, in this case, satisfies
\[
\left<\Rm^*(JX,X)X,JX\right>=4|X|^4,\quad \left<\Rm^*(Y,X)X,Y\right>=|X|^2|Y|^2
\]
for all $X$ in $\mathcal D$ and $Y$ in $\{X,JX\}^\perp\cap\mathcal D$ (see \cite{Sa}). Let $\mu$ be the Riemannian volume.

\begin{prop}\label{maincor3}
The complex Hopf fibration satisfies
\[
\mu(\varphi_t(U))= \int_{U}\frac{(1-t)\mathcal M_1(\mathbf{k_1}(x),t)\mathcal M_2^{2n-2}(\mathbf{k_2}(x),t)}{\mathcal M_1(\mathbf{k_1}(x),0)\mathcal M_2^{2n-2}(\mathbf{k_2}(x),0)}d\mu(x)
\]
for any Borel set $U$, where
\[
\begin{split}
&\mathbf{k_1}(x)=8|\mathfrak f(x)|+(v_0\mathfrak f(x))^2, \quad\mathbf{k_2}(x)=2|\mathfrak f(x)|+\frac{(v_0 \mathfrak f(x))^2}{4},\\
&\mathfrak f(\x)=-\frac{1}{2}d^2_{CC}(\x,\x_0).
\end{split}
\]
In particular, the metric measure space $(M, d_{CC}, \mu)$ satisfies the condition $\hMCP{4}{1}{2n}{2n+1}$ and hence $\MCP 0{2n+3}$.
\end{prop}

Finally, in the case of negative curvature, we have the anti de-Sitter space. In this case, the manifold $M$ is $
\mathbb H^{2n+1}=\{z\in\Complex^{n+1}||z|_H=1\}$, where $|z|_H$ is defined by
\[
|z|_H=|z_{n+1}|^2-|z_1|^2-...-|z_n|^2.
\]
The vector field $v_0$ is given by $v_0=i\sum_{j=1}^{n+1}(z_j\partial_{z_j}-\bar z_j\partial_{\bar z_j})$. The quotient of $\mathbb H^{2n+1}$ by the flow of $v_0$ is the complex hyperbolic space $\mathbb C\mathbb H^n$.

The Riemannian metric coincides with the Euclidean one on $\mathcal D$ and $J$ coincides with standard complex structure on $\mathcal D$. The Tanaka-Webster curvature $\Rm^*$, in this case, satisfies
\[
\left<\Rm^*(JX,X)X,JX\right>=-4|X|^4,\quad \left<\Rm^*(Y,X)X,Y\right>=-|X|^2|Y|^2
\]
for all $X$ in $\mathcal D$ and $Y$ in $\{X,JX\}^\perp\cap\mathcal D$ (see \cite{Ep} or \cite{Wa}). Let $\mu$ be the Riemannian volume.

\begin{prop}\label{maincor4}
The anti de-Sitter space $\mathbb H^{2n+1}$ satisfies
\[
\mu(\varphi_t(U))= \int_{U}\frac{(1-t)\mathcal M_1(\mathbf{k_1}(x),t)\mathcal M_2^{2n-2}(\mathbf{k_2}(x),t)}{\mathcal M_1(\mathbf{k_1}(x),0)\mathcal M_2^{2n-2}(\mathbf{k_2}(x),0)}d\mu(x)
\]
for any Borel set $U$, where
\[
\begin{split}
&\mathbf{k_1}(x)=-8|\mathfrak f(x)|+(v_0\mathfrak f(x))^2, \quad\mathbf{k_2}(x)=-2|\mathfrak f(x)|+\frac{(v_0 \mathfrak f(x))^2}{4},\\
&\mathfrak f(\x)=-\frac{1}{2}d^2_{CC}(\x,\x_0).
\end{split}
\]
In particular, the metric measure space $(M, d_{CC}, \mu)$ satisfies the condition $\hMCP{-4}{-1}{2n}{2n+1}$.
\end{prop}

Our approach gives very sharp result in the sense that Propositions \ref{maincor2}, \ref{maincor3}, and \ref{maincor4} follow from the proof of Theorem \ref{main}.

\smallskip

\section{Sub-Riemannian Geodesic Flows and Curvature on Sasakian manifolds}

In this section, we recall the definition of the sub-Riemannian geodesic flow and its connections with the contraction of measures appeared in \cite{AgLe1,AgLe2}.

As in the Riemannian case, the (constant speed) minimizers of (\ref{SRd}) can be found by minimizing the following kinetic energy functional
\begin{equation}\label{energy}
\inf_{\gamma\in\Gamma}\int_0^1\frac{1}{2}|\dot\gamma(t)|^2dt.
\end{equation}
In the Riemannian case, the minimizers of (\ref{energy}) are given by the geodesic equation. In the sub-Riemannian case, the minimization problem in (\ref{energy}) becomes a constrained minimization problem and it is more convenient to look at the geodesic flow from the Hamiltonian point of view in this case. For this, let $\ham:T^*M\to\Real$ be the Hamiltonian defined by the Legendre transform:
\[
\ham(\x, \p)=\sup_{\tv\in \Di}\left(\p(\tv)-\frac{1}{2}|\tv|^2\right).
\]
This Hamiltonian, in turn, defines a Hamiltonian vector field $\vec\ham$ on the cotangent bundle $T^*M$ which is a sub-Riemannian analogue of the geodesic equation. It is given, in the local coordinates $(x_1,...,x_N,p_1,...,p_N)$, by
\[
\vec\ham=\sum_{i=1}^N\left(\ham_{p_i}\partial_{x_i}-\ham_{x_i}\partial_{p_i}\right).
\]
We assume, through out this paper, that the vector field $\vec\ham$ defines a complete flow which is denoted by $e^{t\vec\ham}$. In the Riemannian case, the minimizers of (\ref{energy}) are given by the projection of the trajectories of $e^{t\vec\ham}$ to the manifold $M$. In the sub-Riemannian case, minimizers obtained this way are called normal geodesics and they do not give all the minimizers of (\ref{energy}) in general (see \cite{Mo} for more detailed discussions on this). On the other hand, all minimizers of (\ref{energy}) are normal if the distribution $\Di$ is contact (see \cite{Mo}).

Next, we discuss an analogue of the Jacobi equation in the above Hamiltonian setting. For this, let $\omega$ be the canonical symplectic form of the cotangent bundle $T^*M$. In local coordinates $(x_1,...,x_N,p_1,...,p_N)$, $\omega$ is given by
\[
\omega=\sum_{i=1}^N{dp_i\wedge dx_i}.
\]

Let $\pi : T^*M \rightarrow M$ be the canonical projection and let $\ver$ the vertical sub-bundle of the cotangent bundle $T^*M$ defined by
\[
\ver_{(\x,\p)}=\{v\in T_{(\x,\p)}T^*M| \pi_*(v)=0\}.
\]
Recall that a $n$-dimensional subspace of a symplectic vector space is Lagrangian if the symplectic form vanishes when restricted to the subspace. Each vertical space $\ver_{(\x,\p)}$ is a Lagrangian subspace of the symplectic vector space $T_{(\x,\p)}T^*M$. Since the flow $e^{t\vec\ham}$ preserves the symplectic form $\omega$, it also sends a Lagrangian subspace to another Lagrangian one. Therefore, the following also forms a one-parameter family of Lagrangian subspaces contained in $T_{(\x,\p)}T^*M$
\begin{equation}\label{Jacobi}
\mathfrak J_{(\x,\p)}(t)=e^{-t\vec\ham}_*(\ver_{e^{t\vec\ham}(\x,\p)}).
\end{equation}
This family defines a curve in the Lagrangian Grassmannian (the space of Lagrangian subspaces) of $T_{(\x,\p)}T^*M$ and it is called the Jacobi curve at $(\x,\p)$ of the flow $e^{t\vec\ham}$.

Assume that the distribution is contact. Then we have the following particular case of the results in \cite{LiZe2,LiZe3} on the canonical bundle of moving frames for curves in Lagrangian Grassmannians:

\begin{theor}\label{frame}
Assume that the distribution $\Di$ is contact. Then there exists a one-parameter family of bases $\bigl(E(t), F(t)\bigr)$, where
\[
\begin{split}
&E(t)=(E_1(t),...,E_{2n+1}(t))^T,\\
&F(t)=(F_1(t),...,F_{2n+1}(t))^T
\end{split}
\]
of the symplectic vector space $T_{(\x,\p)}T^*M$ such that the followings hold for any $t$:
\begin{enumerate}
\item $\mathfrak J_{(\x,\p)}(t)={\rm span}\{E_1(t),...,E_{2n+1}(t)\}$,
\item ${\rm span}\{F_1(t),...,F_{2n+1}(t)\}$ is a family of Lagrangian subspaces,
\item $\omega(F_i(t),E_j(t))=\delta_{ij}$,
\item $\dot E(t)=C_1E(t)+C_2F(t)$,
\item $\dot F(t)=-R(t)E(t)-C_1^TF(t)$,
\end{enumerate}
where $R(t)$ is a symmetric matrix, $C_1$ and $C_2$ are $(2n+1)\times(2n+1)$ matrices defined by
\begin{enumerate}
\item [a)] $\tilde C_1=\left(\begin{array}{cc}
0 & 1\\
0 & 0
\end{array}\right)$ is a $2\times 2$ matrix,
\item [b)] $\tilde C_2=\left(\begin{array}{cc}
0 & 0\\
0 & 1
\end{array}\right)$ is a $2\times 2$ matrix,
\item [c)] $C_1=\left(\begin{array}{cc}
\tilde C_1 & O\\
O & O
\end{array}\right)$
\item [d)] $C_2=\left(\begin{array}{cc}
\tilde C_2 & O\\
O & I
\end{array}\right)$.
\end{enumerate}
(note that in the right hand sides of equations in items (4) and (5) these matrices (with scalar entries) are multiplied on $E$ and $F$, which are columns with vector entries resulting again columns with vector entries).

Moreover, a moving frame $\tilde E_1(t),...,\tilde E_{2n+1}(t),\tilde F_1(t),...,\tilde F_{2n+1}(t)$ satisfies conditions (1)-(6) above
if and only if
\begin{equation}\label{u}
 (\widetilde E_1(t),\widetilde E_2(t),\widetilde
F_1(t),\widetilde F_2(t))=\pm(E_1(t),E_2(t),F_1(t),F_2(t))
\end{equation}
and there exists a constant orthogonal matrix $U$ of size $(2n-1)\times (2n-1)$ (independent of time $t$) such that
\begin{equation}
\widetilde E_{i}(t)=\sum_{j=3}^{2n+1}U_{i-2, j-2}E_{j}(t)\text{ and }\widetilde F_{i}(t)=\sum_{j=3}^{2n+1}U_{i-2,j-2}F_{j}(t)
\end{equation}
for all $i=3,\ldots, 2n+1$  and any time $t\geq 0$.
\end{theor}

We call any frame $(E(t), F(t))$ in Theorem \ref{frame} a canonical frame at the point $(\x,\p)$ and call the equations in (5) and (6) of Theorem \ref{frame} the structural equation of the Jacobi curve (\ref{Jacobi}). Note that the conditions (1) - (3) means that the canonical frame is a family of symplectic bases.

Let $\mathcal R:\ver\to\ver$ be the operator defined by
\[
\mathcal R(E_i(0))=\sum_{j=1}^{2n+1}R_{ij}(0)E_j(0).
\]
Under the following identification, we can also consider $\mathcal R$ as an operator on $TM$
\[
v\in T_{\x}M\mapsto \alpha(\cdot)=\left<v,\cdot\right>\in T^*_{\x}M\mapsto \frac{d}{dt}(\p+t\alpha)\Big|_{t=0}\in\ver_{(\x,\p)}.
\]
Let $\ver_1={\rm span}\{E_1(0)\}$, $\ver_2={\rm span}\{E_2(0)\}$, and
\[
\ver_3={\rm span}\{E_3(0),...,E_{2n+1}(0)\}.
\]
The proofs of following results can be found in \cite{LeLi}.

\begin{theor}
Under the above identification of $\ver_{(\x,\p)}$ and $T_{\x}M$, we have the followings
\begin{enumerate}
\item $\ver_1=\Real v_0$,
\item $\ver_2=\Real J \p^h$,
\item $\ver_3=\Real (\p^h+\p(v_0)v_0)\oplus\{v\in T_\x M|\left<v,\p^h\right>=\left<v,J \p^h\right>=\left<v,v_0\right>=0\}$,
\end{enumerate}
where $\p^h$ is the tangent vector in the distribution $\mathcal D$ defined by $\p(\cdot)=\left<\p^h,\cdot\right>$ on $\Di$.
\end{theor}

\begin{theor}\label{compute}
Assume that the manifold is Sasakian. Then, under the above identifications, $\mathcal R$ is given by
\begin{enumerate}
\item $\mathcal R(v)=0$ for all $v$ in $\ver_1$,
\item $\mathcal R(v)=\Rm^*(v,\p^h)\p^h +\p(v_0)^2v$ for all $v$ in $\ver_2$,
\item $\mathcal R(\p^h+\p(v_0)v_0)=0$,
\item $\mathcal R(v)_{\ver_3}=\Rm^*(v^h,\p^h)\p^h+\frac{1}{4}p(v_0)^2v^h$ for all $v$ in $\ver_3$ satisfying $\left<v^h,\p^h\right>=0$,
\end{enumerate}
where $\p^h$ and $v^h$ are the tangent vectors in the distribution $\mathcal D$ defined by $\p(\cdot)=\left<\p^h,\cdot\right>$ and $\left<v,\cdot\right>=
\left<v^h,\cdot\right>$ on $\Di$, respectively,  and $\mathcal R(v)_{\ver_i}$ denotes the projection of $\mathcal R(v)$ to $\ver_i$ with respect to the splitting $$TM=\ver_1\oplus \ver_2\oplus\ver_3.$$
\end{theor}

\smallskip

\section{Measure Contraction and Matrix Riccati Equation}
\label{Ricattisec}

In this section, we discuss the connections between measure contraction properties and the matrix Riccati equations. For this, first let us summarize the regularity properties of the distance function $\x\mapsto d_{CC}(\x,\x_0)$ and how they affect the properties of  length-minimizing paths in the following:

\begin{lemma}
\label{clar1.3}
For a contact sub-Riemannian metric the following statements hold:
\begin{enumerate}
\item
For any point $\x_0$ in the manifold $M$, the distance function $\x\mapsto d_{CC}(\x,\x_0)$ is locally semi-concave in $M-\{\x_0\}$. In particular, it is twice differentiable Lebesgue almost everywhere in $M$.

\item
For any point $\x$ for which  the function $\x\mapsto d_{CC}(\x,\x_0)$ is differentiable, 
there exists a unique length minimizing sub-Riemannian geodesic and this geodesic is normal.
\end{enumerate}
\end{lemma}

Both items of Lemma \ref{clar1.3} are direct consequences of the results of \cite{CaRi}: the first statement of item (1) regarding semi-concavity follows from Theorem 5 or more general Theorem 1 there and the fact that there is no singular curve in the contact case (see \cite{CaSi, Mo} for more detail). The second statement of item (1) uses Alexandrov's theorem \cite[Theorem 2.3.1 (i)]{CaSi} that a locally semi-concave function is twice differentiable almost everywhere. Item (2) of Lemma \ref{clar1.3} follows  from Theorem 4 of \cite{CaRi}.

Let $\mathfrak f(\x)=-\frac{1}{2}d^2_{CC}(\x,\x_0)$. By item (2) of Lemma \ref{clar1.3},  we can define the family of Borel maps $\varphi_t:M\to M$ so that $t\mapsto\varphi_t(\x)$ is a minimizing geodesic between the points $\x$ and $\x_0$ for almost every $\x$.
\begin{lemma}\label{clar1.32}
For each $t$ in $[0,1)$, there is a set $U$ of full measure such that the map $\varphi_t$ is differentiable and injective on $U$.
\end{lemma}

\begin{proof}
Let $U$ be the set where $\mathfrak f$ is twice differentiable. It follows as in \cite{AgLe1} that the map $\varphi_t$ is given by $\varphi_t(x)=\pi(e^{t\vec H}(d\mathfrak f_x))$ on $U$. Indeed, for each horizontal curve $\gamma:[0,1]\to M$ satisfying $\gamma(1)=\x_0$, we have
\begin{equation}\label{bolza}
\frac{1}{2}\int_0^1|\dot\gamma(t)|^2dt+\mathfrak f(\x)\geq 0.
\end{equation}
Equality holds if and only if $\gamma$ is a length minimizing geodesic connecting $\x$ and $\x_0$. Therefore, $\gamma(t)=\varphi_t(\x)$ if $\gamma$ is a minimizer of the functional (\ref{bolza}). The rest follows from Pontryagin maximum principle for Bolza problem (see \cite{CaSi}). This also gives a proof of item (2) in Lemma \ref{clar1.3}.

Since $t\mapsto\varphi_t(\x)$ is a minimizing geodesic which starts from $\x$ and ends at $\x_0$ and different normal minimizing geodesics which end at the same point $x_0$ don't intersect except possibly at the end points, it follows that $\varphi_t$ is injective on $U$ for each $t$ in $[0,1)$.
\end{proof}
Further, let $\mu$ be the Riemannian volume form.  As follows from \cite{FiRi}, the measures $(\varphi_t)_*\mu$ are absolutely continuous with respect to $\mu$ for all time $t$ in the interval $[0,1)$. If $(\varphi_t)_*\mu=\rho_t\mu$, then the following (Jacobian) equation holds on a set of full measure where $\mathfrak f$ is twice differentiable:
\begin{equation}
\label{Jacobianeq}
\rho_t(\varphi_t(\x))\det((d\varphi_t)_\x)=1
\end{equation}
and the determinant is computed with respect to frames of the above mentioned Riemannian structure. Here we use the fact that in our situation the classical formula for the change of variables in integrals is still valid although the maps $\varphi_t$ representing the changes of variables are just differentiable almost everywhere (see, for example, \cite[Theorem 11.1 (ii)]{Villani}).

Besides, since, by Lemma \ref{clar1.32}, the maps $\varphi_t$ are invertible for all $t$ in $[0,1)$, we have
\begin{equation}\label{vol}
\mu(\varphi_t(U))=\int_U\frac{1}{\rho_t(\varphi_t(\x))}d\mu(\x)=\int_U\det((d\varphi_t)_\x)d\vol(\x).
\end{equation}
Therefore, in order to prove the main results and the measure contraction properties, it remains to estimate $\det((d\varphi_t)_\x)$ which can be done using the canonical frame mentioned above. The explanations on this will occupy the rest of this section.

Let $\x$ be a point where the function $\mathfrak f$ is twice differentiable and let $(E(t), F(t))$ be a canonical frame at the point $(\x,d\mathfrak f_\x)$. Let $\varsigma_i=\pi_*(F_i(0))$ be the projection of the frame $F(0)$ onto the tangent bundle $TM$. Let $dd\mathfrak f$ be the differential of the map $\x\mapsto d\mathfrak f_{\x}$ which pushes the above frame on $T_{\x}M$ to a tuple of vectors in $T_{(\x,d\mathfrak f)}T^*M$. Therefore, we can let $A(t)$ and $B(t)$ be the matrices defined by
\begin{equation}\label{AB1}
dd\mathfrak f(\varsigma) =A(t)E(t)+B(t)F(t),
\end{equation}
where $\varsigma=\left(\varsigma_1,\cdots,\varsigma_{2n+1}\right)^T $ and $dd\mathfrak f(\varsigma)$ is the column obtained by applying $dd\mathfrak f$ to each entries of $\varsigma$.
By (\ref{AB1}) and the definition of $\varphi_t$, we have
\begin{equation}\label{conjugate}
d\varphi_t(\varsigma)=B(t)(\pi_* e^{t\vec\ham}_*F(t)).
\end{equation}
Note that $S(t)$ is defined for all $t$ in $(0,1)$. Indeed, $S$ is defined by (\ref{AB1}) and $S(t)=B(t)^{-1}A(t)$ (not as a solution of a differential equation). Since the geodesic $\varphi_t(\x)$ is minimizing, there is no conjugate point along it (see \cite{AgSa}). By (\ref{conjugate}) and $d\varphi_1=0$, it follows that $B(t)$ is invertible and $S(t)$ is defined for all $t$ in $[0,1)$. Besides, the tangent spaces to the set $\{(\x,d\mathfrak f_\x):\x\in M\}$ at points of smoothness are Lagrangian. This together with \eqref{AB1} implies that the matrix $S(t)$ is symmetric.

\begin{lemma}\label{keylem}
Let $S(t)=B(t)^{-1}A(t)$. Then
\[
\mu(\varphi_t(U))= \int_U e^{-\int_0^t\tr(S(\tau)C_2)d\tau}d\mu(\x).
\]

Moreover, $S(t)$ satisfies the following matrix Riccati equation
\[
\dot S(t)-S(t)C_2S(t)+C_1^TS(t)+S(t)C_1-R(t)=0,\quad \lim_{t\to 1}S(t)^{-1}=0.
\]
\end{lemma}

\begin{proof}
Note that $\tau\mapsto e^{t\vec\ham}_*F(t+\tau)$ is a canonical frame at $e^{t\vec\ham}(\x,d\mathfrak f)$. Therefore, by \cite[Theorem 4.3]{LeLi}, we have
\[
\begin{split}
&d_{CC}(\x_0, \x)|\det(d\varphi_t)|=|\nabla_\hor\mathfrak f(\x)||\det(d\varphi_t)|\\
&= |\vol(d\varphi_t(\varsigma))|=|\det(B(t))\vol(\pi_* e^{t\vec\ham}_*F(t))|\\
&= |\nabla_\hor\mathfrak f(\x)||\det(B(t))|= d_{CC}(\x_0,\x)|\det(B(t))|.
\end{split}
\]

Here $\nabla_\hor\mathfrak f$ denotes the horizontal gradient of $\mathfrak f$ defined by $d\mathfrak f(v)=\left<\nabla_\hor\mathfrak f,v\right>$, where $v$ is any vector in the distribution $\Di$.

By combining this with \eqref{vol}, we obtain
\begin{equation}\label{vol2}
\mu(\varphi_t(U))= \int_U |\det(B(t))|d\vol(\x).
\end{equation}

On the other hand, by differentiating \eqref{AB1} with respect to time $t$ and using the structural equation, we obtain
\begin{equation}\label{AB}
\dot A(t)+A(t)C_1-B(t)R_t=0,\quad \dot B(t)+A(t)C_2-B(t)C_1^T=0.
\end{equation}
Therefore,
\[
\frac{d}{dt}\det(B(t))=\det(B(t))\tr(B(t)^{-1}\dot B(t))=-\det(B(t))\tr(S(t)C_2).
\]

By setting $t=0$ and applying  $\pi_*$ on each side of \eqref{AB1}, we have $B(0)=I$. Therefore, we obtain
\[
\det(B(t))=e^{-\int_0^t\tr(S(\tau)C_2)d\tau}.
\]
By combining this with (\ref{vol2}), we obtain the first assertion.

Since $\varphi_1(\x)=\x_0$ for all $\x$, we have $d\varphi_1=0$ and so $B(1)=0$.
This together with \eqref{AB1} inplies that the matrix $A(1)$ is invertible , therefore $S(t)$ is invertible for all $t<1$ sufficiently close to $1$.
By \eqref{AB} and the definition of $S(t)$, we have
\[
\dot S(t)-S(t)C_2S(t)+C_1^TS(t)+S(t)C_1-R(t)=0,\quad \lim_{t\to 1}S(t)^{-1}=0
\]
as claimed.
\end{proof}

\smallskip

\section{Estimates for Solutions of Matrix Riccati Equations}
\label{Ricsec}

According to Lemma \ref{keylem}, we need to estimate the term $\tr(S(t)C_2)$. In this section, we provide two such estimates which lead to the main results.

Throughout this section, we assume that the matrix $R(t)$ from item (5) of Theorem \ref{frame} is of the form
\begin{equation}
\label{structmat}
R(t)=\left(\begin{array}{cccc}
0 & 0 &  O_{2n-2} & 0\\
0 & R_{bb}(t) &  R_{cb}(t) & 0\\
O_{2n-2}^T &  R_{cb}(t)^T & R_{cc}(t) & O_{2n-2}^T\\
0 & 0 & O_{2n-2} & 0
\end{array}
\right),
\end{equation}
where blocks in this matrix are chosen according to the following partition of the tuple of vectors  $E(t)$: $E(t)=(E_a(t), E_b(t), E_c(t)$ with
\[
E_a(t)=E_1(t), E_b(t)=E_2(t), E_c(t)=\bigl(E_3(t),\ldots, E_{2n+1}(t)\bigr)
\]
so that $E_{2n+1}(0)=\p^h+\p(v_0)v_0$ and
$O_{2n-2}$ is the zero matrix of size $1\times(2n-2)$. The specific form of this matrix (i.e. the presence of zero blocks) comes from items (1) and (3)  of Theorem \ref{compute}.

The following is a consequence of the result in \cite{Ro}.

\begin{lemma}\label{Riccati1}
Assume that the curvature $R(t)$ satisfies
\[
\left(\begin{array}{cc}
R_{bb}(t) &  R_{cb}(t)\\
R_{cb}(t)^T & R_{cc}(t)
\end{array}
\right)\geq \left(\begin{array}{cc}
\mathfrak{k_1} &  0\\
0 & \mathfrak{k_2}I_{2n-2}\\
\end{array}
\right),
\]
where $\mathfrak{k_1}$ and $\mathfrak{k_2}$ are two constants. Then
\[
e^{-\int_0^t\tr(C_2S(\tau))d\tau}\geq (1-t)^{2n+3}\frac{\mathcal M_1(\mathfrak{k_1},t)\mathcal M_2^{2n-2}(\mathfrak{k_2},t)}{\mathcal M_1(\mathfrak{k_1},0)\mathcal M_2^{2n-2}(\mathfrak{k_2},0)},
\]
where
\[
\begin{split}
&\mathfrak D(k,t)=\sqrt{|k|}(1-t),\\
&\mathcal M_1(k,t)=\begin{cases}
\frac{2-2\cos(\mathfrak D(k,t))-\mathfrak D(k,t)\sin(\mathfrak D(k,t))}{\mathfrak D(k,t)^4} & \text{if } k>0\\
\frac{1}{12} & \text{if } k=0\\
\frac{2-2\cosh(\mathfrak D(k,t))+\mathfrak D(k,t)\sinh(\mathfrak D(k,t))}{\mathfrak D(k,t)^4} & \text{if } k<0,\\
\end{cases}\\
&\mathcal M_2(k,t)=\begin{cases}
\frac{\sin(\mathfrak D(k,t))}{\mathfrak D(k,t)} & \text{if } k>0\\
1 & \text{if } k=0\\
\frac{\sinh(\mathfrak D(k,t))}{\mathfrak D(k,t)} & \text{if } k<0.
\end{cases}
\end{split}
\]
Moreover, equality holds if
\[
\left(\begin{array}{cc}
R_{bb}(t) &  R_{cb}(t)\\
R_{cb}(t)^T & R_{cc}(t)
\end{array}
\right)= \left(\begin{array}{cc}
\mathfrak{k_1} &  0\\
0 & \mathfrak{k_2}I_{2n-2}\\
\end{array}
\right),
\]
\end{lemma}

\begin{proof}
We only prove the case when both constants  $\mathfrak{k_1}$ and $\mathfrak{k_2}$ are positive. The proofs for other cases are similar and are therefore omitted. In the sequel we use the natural partial order on the space of symmetric matrices: $S_1\geq S_2$ if the symmetric matrix $S_1-S_2$ corresponds to a nonnegative definite quadratic form.

As was already mentioned before Lemma \ref{keylem},  the  $S(t)=B(t)^{-1}A(t)$ is defined for all $t$ in $[0,1)$.

Recall that $S(t)$ satisfies
\[
\dot S(t)-S(t)C_2S(t)+C_1^TS(t)+S(t)C_1-R(t)=0,\quad \lim_{t\to 1}S(t)^{-1}=0.
\]
We will estimate $S(t)$ using the comparison theorems for the matrix Riccati equation from \cite{Ro}. Below, we give the detail of this estimate.

Indeed, by (\ref{AB}), $\tilde A(t):=A(1-t)$ and $\tilde B(t):=B(1-t)$ satisfy
\[
\dot{\tilde A}(t)=\tilde A(t)C_1-\tilde B(t)\tilde R(t),\quad \dot{\tilde B}(t)=\tilde A(t)C_2-\tilde B(t)C_1^T,
\]
where $\tilde R(t)=R(1-t)$.

Note that $\tilde B(0)=0$ and $\tilde A(0)$ is invertible. A computation shows that
\[
\tilde A(t)=\tilde A(0)\left(I+tC_1-\frac{t^2}{2}C_2\tilde R(0)+O(t^3)\right)
\]
and
\[
\tilde B(t)=\tilde A(0)\left(tC_2+\frac{t^2}{2}(C_1-C_1^T)-\frac{t^3}{6}\left(C_1C_1^T+C_2\tilde R(0)C_2\right)+O(t^4)\right).
\]

Let $U(t)=S(1-t)^{-1}=\tilde A(t)^{-1}\tilde B(t)$. It follows that
\[
U(t)=tC_2-\frac{t^2}{2}(C_1+C_1^T)+\frac{t^3}{3}(C_1C_1^T+C_2\tilde R(0)C_2)+O(t^4).
\]
Note that $tC_2-\frac{t^2}{2}(C_1+C_1^T)+\frac{t^3}{3}(C_1C_1^T+C_2\tilde R(0)C_2)\geq\frac{t^3}{12}I$ for all small enough $t$. Moreover, $U$ satisfies the following matrix Riccati equation
\[
\dot U(t)-U(t)\tilde R(t) U(t)+C_1U(t)+U(t)C_1-C_2=0.
\]

We compare this with the equation
\begin{equation}\label{compare}
\dot \Lambda(t)-\Lambda(t)K \Lambda(t)+C_1\Lambda(t)+\Lambda(t)C_1-C_2=0,
\end{equation}
where $K=\left(\begin{array}{cccc}
0 & 0 & 0 & 0\\
0 & \mathfrak{k_1} & 0 & 0\\
0 & 0 & \mathfrak{k_2}I_{2n-2} & 0\\
0 & 0 & 0 & 0\\
\end{array}
\right)$.

A computation using the result in \cite{Le} shows that
\[
\Lambda(t)=\left(
             \begin{array}{cccc}
               -\frac{t}{\mathfrak k_1}+\frac{\tan(\sqrt{\mathfrak k_1}t)}{\mathfrak k_1^{3/2}} & \frac{1}{\mathfrak k_1}-\frac{\sec(\sqrt{\mathfrak k_1}t)}{\mathfrak k_1} & 0 & 0 \\
               \frac{1}{\mathfrak k_1}-\frac{\sec(\sqrt{\mathfrak k_1}t)}{\mathfrak k_1} & \frac{\tan(\sqrt{\mathfrak k_1}t)}{\sqrt{\mathfrak k_1}} & 0 & 0 \\
               0 & 0 & \frac{\tan(\sqrt{\mathfrak k_2}t)}{\sqrt{\mathfrak k_2}}I_{2n-2} & 0 \\
               0 & 0 & 0 & t \\
             \end{array}
           \right)
\]
satisfies (\ref{compare}). Note also that
\[
\Lambda(t)=tC_2-\frac{t^2}{2}(C_1+C_1^T)+\frac{t^3}{3}(C_1C_1^T+C_2KC_2)+O(t^4).
\]
Next, we apply \cite[Theorem 1]{Ro} to $U(t)$ and $\Lambda(t+\epsilon)$, then let $\epsilon$ goes to zero. We obtain
\[
U(t)\geq \Lambda(t)
\]
for all small enough $t$ for which both $U$ and $\Lambda$ are defined. For the same set of $t$, we also have
\begin{equation}\label{compare2}
S(1-t)\leq \Gamma(1-t),
\end{equation}
where $\Gamma(1-t)=\Lambda(t)^{-1}$. By applying \cite[Theorem 1]{Ro} again, we obtain (\ref{compare2}) for all $t$ in $[0,1]$.

Since the product of non-negative definite symmetric matrices has non-negative trace, it follows that $\tr(C_2S(t))\leq \tr(C_2 \Gamma(t))$. The assertion follows from integrating the above inequality.
\end{proof}

Next, we consider the case where the assumptions are weaker than those in Lemma \ref{Riccati1}.

\begin{lemma}\label{Riccati2}
Assume that the curvature $R(t)$ satisfies $R_{bb}(t)\geq \mathfrak k_1$ and $\tr(R_{cc}(t))\geq (2n-2)\mathfrak k_2$ for some constants $\mathfrak{k_1}$ and $\mathfrak{k_2}$. Then
\[
e^{-\int_0^t\tr(C_2S(\tau))d\tau}\geq (1-t)^{2n+3}\frac{\mathcal M_1(\mathfrak{k_1},t)\mathcal M_2^{2n-2}(\mathfrak{k_2},t)}{\mathcal M_1(\mathfrak{k_1},0)\mathcal M_2^{2n-2}(\mathfrak{k_2},0)}.
\]
\end{lemma}

\begin{proof}
Once again, we only prove the case when both constants  $\mathfrak{k_1}$ and $\mathfrak{k_2}$ are positive. Let us write
\[
S(t)=\left(\begin{array}{ccc}
S_1(t) & S_2(t) & S_3(t)\\
S_2(t)^T & S_4(t) & S_5(t)\\
S_3(t)^T & S_5(t)^T & S_6(t)
\end{array}\right),
\]
where $S_1(t)$ is a $2\times 2$ matrix and $S_6(t)$ is $1\times 1$. Then
\begin{equation}\label{Seqn}
\begin{split}
&\dot S_1(t)-S_1(t)\tilde C_2S_1(t)-S_2(t)S_2(t)^T\\
&-S_3(t)S_3(t)^T+\tilde C_1^TS_1(t)+S_1(t)\tilde C_1-R_1(t)=0,\\
&\dot S_4(t)-S_4(t)^2-S_5(t)S_5(t)^T-S_2(t)^T\tilde C_2S_2(t)-R_{cc}(t)=0,\\
&\dot S_6(t)-S_6(t)^2-S_5(t)^TS_5(t)-S_3(t)^T\tilde C_2S_3(t)=0,
\end{split}
\end{equation}
where $\tilde C_1=\left(\begin{array}{cc}
0 & 1 \\
0 & 0 \\
\end{array}
\right)$,\ $\tilde C_2=\left(\begin{array}{cc}
0 & 0 \\
0 & 1 \\
\end{array}
\right)$, and $R_1(t)=\left(\begin{array}{cc}
0 & 0\\
0 & R_{bb}(t)\\
 \end{array}
\right)$.

By the same argument as in Lemma \ref{Riccati1}, we have
\begin{equation}\label{split1}
\begin{split}
\tr(\tilde C_2S_1(t))&\leq \frac{\sqrt{\mathfrak k_1}(\sin(\mathfrak D(\mathfrak k_1,t))-\mathfrak D(\mathfrak k_1,t)\cos(\mathfrak D(\mathfrak k_1,t)))}{2-2\cos(\mathfrak D(\mathfrak k_1,t))-\mathfrak D(\mathfrak k_1,t)\sin(\mathfrak D(\mathfrak k_1,t))}.
\end{split}
\end{equation}

For the term $S_4(t)$, we can take the trace and obtain
\[
\begin{split}
\frac{d}{dt} \tr(S_4(t))&\geq \frac{1}{2n-2}\tr(S_4(t))^2+(2n-2)\mathfrak k_2.
\end{split}
\]

Therefore, an argument as in Lemma \ref{Riccati1} again gives
\begin{equation}\label{split2}
\tr S_4(t)\leq \sqrt{|\mathfrak k_2|}(2n-2)\cot\left(\mathfrak D(\mathfrak k_2,t)\right).
\end{equation}

Finally, for the term $S_6(t)$, we also have
\[
\dot S_6(t)\geq S_6(t)^2.
\]
Therefore,
\[
S_6(t)\leq \frac{1}{1-t}.
\]

By combining this with (\ref{split1}) and (\ref{split2}), we obtain
\[
\begin{split}
\tr(C_2S(t))&\leq \sqrt{|\mathfrak k_2|}(2n-2)\cot\left(\mathfrak D(\mathfrak k_2,t)\right)+\frac{1}{1-t}\\
&+\frac{\sqrt{\mathfrak k_1}(\sin(\D(\mathfrak k_1,t))-\D(\mathfrak k_1,t)\cos(\D(\mathfrak k_1,t)))}{2-2\cos(\D(\mathfrak k_1,t))-\D(\mathfrak k_1,t)\sin(\D(\mathfrak k_1,t))}.
\end{split}
\]
The rest follows as in Lemma \ref{Riccati1}.
\end{proof}

Finally, we finish the proof of the main result. Note that Propositions \ref{maincor2},  \ref{maincor3}, and \ref{maincor4} are consequences of the proof of Theorem \ref{main} and equality case of Lemma \ref{Riccati1}.

\begin{proof}[Proof of Theorem \ref{main}]
If $E(t), F(t)$ is a canonical frame at the point $(\x,d\mathfrak f_{\x})$ in the cotangent bundle $T^*M$, then
\[
t\mapsto (e^{\tau\vec \ham}_*(E(t+\tau)), e^{\tau\vec \ham}_*(F(t+\tau)))
\]
is a canonical frame at the point $e^{\tau\vec \ham}(\x,d\mathfrak f_{\x})$. It follows from this that $R(t)$ is the matrix representation of the operator $\mathfrak R_{e^{t\vec\ham}(\x,d\mathfrak f)}$ with respect to the frame $e^{\tau\vec \ham}_*(E(\tau))$.

Since $v_0$ is a symmetry, the function $u_0(\x,\p)=\p(v_0(\x))$ is constant along the flow $e^{t\vec\ham}$ (see for instance \cite{Mo}). Therefore, by the assumptions and Theorem \ref{compute}, $R_{bb}(t)$ and $\tr(R_{cc}(t))$ are bounded below by $k_1|\nabla_{\hor}\mathfrak f|^2+u_0^2(\x,d\mathfrak f)$ and $(2n-2)\left(k_2|\nabla_\hor\mathfrak f|^2+\frac{u_0^2(\x,d\mathfrak f)}{4}\right)$, respectively.  Therefore, the assumptions of Lemma \ref{Riccati2} are satisfied. By combining this with Lemma \ref{keylem}, the result follows.
\end{proof}

\smallskip

\section{Appendix}

In this appendix, we give a proof of Corollary \ref{Poincare} which is a minor modification of the one in \cite{LoVi2}.

\begin{proof}[Proof of Corollary \ref{Poincare}]
Let $x'$ and $x_0$ be two points on the manifold $M$. Let $\mathfrak f(x)=-\frac{1}{2}d^2(x,x')$ and let
\[
\varphi_t(x):=\pi(e^{t\vec H}(d\mathfrak f_{x})).
\]

Recall that $t\mapsto\varphi_t(x)$ is a minimizing geodesic connecting $x$ and $x'$ for $\eta$-almost all $x$. Therefore, by the measure contraction property,
\[
\begin{split}
&f(x')-\left<f\right>_{B_{x_0}(R)}\\
&=\frac{1}{\eta(B_{x_0}(R))}\int_{B_{x_0}(R)}f(\varphi_1(x))-f(x)d\eta(x)\\
&= \frac{1}{\eta(B_{x_0}(R))}\int_{B_{x_0}(R)}\int_0^1\left<\nabla_\hor f(\varphi_t(x)),\dot \varphi_t(x)\right>dt d\eta(x)\\
&\leq \frac{kR}{\eta(B_{x_0}(R))}\int_0^1\int_{B_{x_0}(R)}|\nabla_\hor f(\varphi_t(x))|d\eta(x)dt\\
\end{split}
\]
\[
\begin{split}
&= \frac{kR}{\eta(B_{x_0}(R))}\int_0^1\int_{\varphi_t(B_{x_0}(R))}|\nabla_\hor f(x)|d(\varphi_t)_*\eta(x)dt\\
&\leq \frac{kR}{\eta(B_{x_0}(R))}\int_0^1\int_{B_{x_0}(2R)}|\nabla_\hor f(x)|d(\varphi_t)_*\eta(x)dt\\
&\leq \frac{CkR}{\eta(B_{x_0}(R))}\int_{B_{x_0}(2R)}|\nabla_\hor f(x)|d\eta(x).
\end{split}
\]

It follows from this, Jensen's inequality, and Corollary \ref{double} that
\[
\begin{split}
&\int_{B_{x_0}(R)}|f(x')-\left<f\right>_{B_{x_0}(kR)}|^pd\eta(x')\\
&\leq C^pR^p\int_{B_{x_0}(R)}\left(\frac{1}{\eta(B_{x_0}(kR))}\int_{B_{x_0}(2kR)}|\nabla_\hor f(x)|d\eta(x)\right)^pd\eta(x')\\
&\leq C^pR^p\int_{B_{x_0}(2R)}|\nabla_\hor f(x)|^pd\eta(x).
\end{split}
\]

Finally, the result follows from this and \cite{Je}.
\end{proof}

\end{document}